\documentclass[11pt,reqno,oneside]{amsart}
\usepackage[utf8]{inputenc}
\usepackage[english]{babel}
\usepackage[babel]{csquotes}

\usepackage{amsmath}
\usepackage{amsfonts}
\usepackage{amssymb}
\usepackage{amsthm} 
\usepackage{dsfont} 
\usepackage[mathscr]{euscript} 
\usepackage{eurosym}
\usepackage{bbm} 
\usepackage[a4paper]{geometry}
\usepackage{comment} 
\usepackage{paralist} 

\usepackage{graphicx}
\usepackage{xkeyval}
\usepackage{pstricks}
\usepackage{hyperref}
\usepackage{longtable}
\usepackage{tikz}

\hypersetup{
urlcolor=black, 
  menucolor=black, 
  citecolor=black, 
  anchorcolor=black, 
  filecolor=black, 
  linkcolor=black, 
  colorlinks=true,
}
\usepackage{comment}

\usetikzlibrary{matrix}

\newcommand{\lebesgue}{\lambda\mspace{-7mu}\lambda}

\newcommand{\Prob}{\mathds{P}}

\newcommand{\E}{\mathds{E}}
\newcommand{\V}{\mathds{V}}

\newcommand{\R}{\mathbb{R}}

\newcommand{\Z}{\mathbb{Z}}
\newcommand{\N}{\mathbb{N}}

\newcommand{\abs}[1]{|{#1}|} 
\newcommand{\bigabs}[1]{\left|{#1}\right|} 
\newcommand{\one}{\mathds{1}}

\newcommand{\Bcal}{\mathcal{B}}

\newcommand{\Gcal}{\mathcal{G}}

\newcommand{\Ical}{\mathcal{I}}
\newcommand{\Ncal}{\mathcal{N}}

\newcommand{\Lcal}{\mathcal{L}}

\newcommand*{\defeq}{\mathrel{\vcenter{\baselineskip0.5ex \lineskiplimit0pt
                     \hbox{\scriptsize.}\hbox{\scriptsize.}}}%
                     =}

\newcommand{\de}{\text{d}}

\newcommand{\norm}[1]{\|#1\|}

\DeclareMathOperator{\diag}{diag}

\theoremstyle{plain}
\newtheorem{lemma}{Lemma}
\newtheorem{theorem}[lemma]{Theorem}

\newtheorem{corollary}[lemma]{Corollary}
\theoremstyle{definition}

\theoremstyle{remark}

\definecolor{darkblue}{rgb}{.1, 0.1,.8}
\definecolor{darkgreen}{rgb}{0,0.8,0.2}
\definecolor{darkred}{rgb}{.8, .1,.1}

\begin{document}
\title[Interval Type Local Laws]
{Interval Type Local Limit Theorems for Lattice Type Random Variables and Distributions}
\author[Michael Fleermann]{Michael Fleermann}
\author[Werner Kirsch]{Werner Kirsch}
\author[Gabor Toth]{Gabor Toth}
\begin{abstract}
In this paper, we propose a new interpretation of local limit theorems for univariate and multivariate distributions on lattices. We show that -- given a local limit theorem in the standard sense -- the distributions are approximated well by the limit distribution, uniformly on intervals of possibly decaying length. We identify the maximally allowable decay speed of the interval lengths. Further, we show that for continuous distributions, the interval type local law holds without any decay speed restrictions on the interval lengths. We show that various examples fit within this framework, such as standardized sums of i.i.d.\ random vectors or correlated random vectors induced by multidimensional spin models from statistical mechanics.
\end{abstract}
\maketitle

\section{Introduction}
Assume $f:\R\to\R$ is a continuous probability density function with respect to the Lebesgue measure $\lebesgue$ on $(\R,\Bcal)$, where $\Bcal$ denotes the Borel $\sigma$-algebra over $\R$. Define $\mu\defeq f\lebesgue$, then if $\mu_n\to\mu$ weakly, we find that the distribution functions $F_{\mu_n}$ of $\mu_n$ converge uniformly to the distribution function $F_{\mu}$ of $\mu$. A straightforward consequence  of this (e.g. \cite[172]{Fleermann:2019:a}) is that
\begin{equation}
\label{eq:absolutedifference}
\sup_{I\in\Ical(\R)} \abs{\mu_n(I)-\mu(I)} \xrightarrow[n\to\infty]{} 0,
\end{equation}
where if $M\subseteq \R$, we define $\Ical(M)\defeq \{I\subseteq M\,|\, I \text{ is an interval with }\abs{I}>0\}$, where $\abs{I}$ denotes the length of the interval. From \eqref{eq:absolutedifference} it does not, however, follow that $\mu_n$ is \emph{well-approximated} by $\mu$ uniformly over all intervals $I\subseteq\R$. The reason is that for small intervals $I$, that is, \emph{locally}, $\mu_n(I)$ and $\mu(I)$ are both close to zero anyway, hence also their absolute difference. Therefore, even if the supremum in \eqref{eq:absolutedifference} is bounded by $10^{-10}$, $\mu_n(I)$ could well be $10^{5}$ times larger than $\mu(I)$, hence the latter can hardly be considered a good approximation of the former.
With this in mind, an improvement of \eqref{eq:absolutedifference} would be a statement of the form
\begin{equation}
\label{eq:relativedifference}
\sup_{I\in\Ical(\R)} \bigabs{\frac{\mu_n(I)}{\mu(I)}-1} \xrightarrow[n\to\infty]{} 0.
\end{equation}

Obviously, a statement of the form \eqref{eq:relativedifference} only makes sense for intervals $I$ that lie in the support of $\mu$, which is given by the support of $f$. So, if $f>0$ on some interval $[a,b]\subseteq\R$ (and then $f\geq c>0$ on $[a,b]$ for some $c>0$ by continuity) we obtain
\begin{equation}
\label{eq:relativeintervaldifference}
\sup_{I\in\Ical([a,b])} \bigabs{\frac{\mu_n(I)}{\mu(I)}-1} \xrightarrow[n\to\infty]{} 0.
\end{equation}
We call a statement of the form \eqref{eq:relativeintervaldifference} an \emph{interval type local limit theorem}, since it allows to conclude relative approximation on smaller and smaller intervals. Note that \eqref{eq:relativeintervaldifference} allows us to dynamically choose a sequence of intervals $[a,b]\supseteq I_1\supseteq I_2\supseteq \ldots$ and still find $\mu_n(I_n)/\mu(I_n) \to 1$ as $n\to\infty$.

In probability theory, many interesting weak convergence results pertain to weak convergence of (the distributions of) appropriately normalized sums, such as in the central limit theorem. A second class of prominent weak convergence results pertain to weak convergence of empirical distributions, such as Wigner's semicircle law in random matrix theory, which postulates convergence of the empirical distribution of eigenvalues of random matrices. There, local limit theorems related to \eqref{eq:relativeintervaldifference} (see \cite{Erdos:Knowles:Yau:Yin:2013:b}, \cite{Tao:Vu:2012} and the relation to \eqref{eq:relativeintervaldifference} in \cite{Fleermann:2019:a}) led to breakthroughs in the analysis of the eigenvalue behavior of random matrices. Since empirical distributions are random discrete distributions, a statement as in \eqref{eq:relativeintervaldifference} cannot hold (almost surely, say), since we can always find a gap between realizations of random variables, and for an interval $I$ within this gap it holds $\mu_n(I)=0$. Therefore, a minimal interval length $m_n$ is required for \eqref{eq:relativeintervaldifference} to hold, and it is an interesting question how fast $m_n$ is allowed to converge to zero. For example, in the local version of Wigner's semicircle law, this question is not answered tightly, yet.

The interval type formulation within the setting of random matrices (\cite[170]{Fleermann:2019:a}) motivated the current work, in which we return to the first class of weak convergence results mentioned above, which pertain to normalized sums of random variables.  In this classical setting, (standard) local limit theorems are well-known, see \cite{Durrett:2019, Petrov:1975, Gnedenko:Kolmogorov:1968} or the famous Moivre-Laplace theorem as in \cite{Feller:1968}. The driving factor that led us to state a multivariate version of our results was motivated by recent findings in \cite{Fleermann:Kirsch:Toth:2020:a}, where a multivariate local limit theorem was shown, where limit theorems in \cite{Kirsch:Toth} are used.

In this paper, we introduce the interval type formulation to the classical setting. We show that for a statement as in \eqref{eq:relativeintervaldifference}, a minimal interval length is required for the case of discrete random variables (and we identify this minimal rate), while  \eqref{eq:relativeintervaldifference} holds as is for continuous random variables. Our results do not only cover the univariate case, but also multivariate settings. In addition, since we take the local limit theorem as a starting point to show the interval type local limit theorem, no requirements on the correlations of the random variables are made. In particular, in our Corollary~\ref{cor:CW} we treat multivariate correlated random variables, for which a (standard) local limit theorem has recently been derived \cite{Fleermann:Kirsch:Toth:2020:a}.

\section{Setup and Results}

We consider a sequence of probability measures $(\mu_n)_n$ on $\R^d$, where we assume that for each $n\in\N$, $\mu_n$ has a grid distribution on $\Gcal(n)\defeq v(n)+w(n)\circ\Z^d$, where $v(n)\in\R^d$ and $w(n)\in\R^d_{>0}$ for all $n\in\N$, and where $\circ$ denotes the Hadamard product, that is, componentwise multiplication. The parameter $w(n)$ is called \emph{width} of the grid $\Gcal(n)$. 
Denote by $\lebesgue^d$ the Lebesgue measure on $\R^d$, and by $f:\R^d\to\R$ a continuous probability density function, that is, $f\geq 0$ and $\int f\de\lebesgue^d =1$. We say that $\mu_n$ converges weakly to $\mu\defeq f\lebesgue$, if for all continuous and bounded functions $g:\R^d\to\R$ it holds that
\[
\int_{\R^d} g \de\mu_n \xrightarrow[n\to\infty]{} \int_{\R^d}g\de\mu.
\]
Further, we say that $\mu_n$ converges locally weakly to $\mu$, if 
\begin{equation}
\label{eq:locallaw}
\sup_{x\in\Gcal(n)}\bigabs{\frac{1}{\bar{w}(n)}\mu_n(x)-f(x)}\xrightarrow[n\to\infty]{} 0,
\end{equation}
where $\bar{w}(n)\defeq \prod w_{\delta}(n)$, and where $\mu_n(x)\defeq\mu_n(\{x\})$ for all $x\in\R^d$.

For our first theorem, we assume that \eqref{eq:locallaw} holds, which takes the form of a local limit theorem in applications, see Corollaries~\ref{cor:iidunivariate} and~\ref{cor:CW}. Usually, for obtaining \eqref{eq:locallaw} one needs additional assumptions on the width $w(n)$, namely that $w(n)\to 0 \in\R^d$ as $n\to\infty$ and that the width is chosen as large as possible, which means that for all $n\in\N$ and $\delta\in\{1,\ldots,d\}$, there is $g(\delta)\in\Gcal(n)$ such that $\mu_n(g(\delta))\cdot\mu_n(g(\delta)+w_{\delta}(n)\cdot e_{\delta}) > 0$, where $e_{\delta}$ denotes the $\delta$-th unit vector in $\R^d$. In the proof of our first theorem, though, we do not formally need these requirements.

The setup we consider is very general, but we have two specific applications in mind that motivated our analysis:
\begin{enumerate}
\item Consider $d$-dimensional random vectors $(X^{(n)}_i)_{1\leq i \leq n<\infty}$, where 
    \[
    X^{(n)}_i = \left(X^{(n)}_{i,1},\ldots,X^{(n)}_{i,d}\right),
    \] 
    with a multidimensional lattice distribution. That is, there exist real vectors $v\geq 0$, $w>0$, such that each $X^{(n)}_i$ assumes values on the grid $\Gcal\defeq v+w\circ\Z^d$. To allow for a different scaling in each dimension, let $c(n)$ be a $d$-dimensional scaling vector, for example $c_{\delta}(n)=n^{-1/2}$ for all $\delta\in\{1,\ldots,d\}$.  Then define
\[
S_n\defeq \sum_{i=1}^n X^{(n)}_i \qquad \text{and} \qquad S_n^{*}\defeq c(n)\circ S_n.
\]
Note that $S_n^{*}$ assumes values on the grid 
\[
\Gcal(n) = v(n) + w(n)\circ\Z^d,
\]
where 
\[
v(n)=n\cdot v\circ c(n)\quad\text{and}\quad w(n)=w\circ c(n).  
\]
Here, the discrete distribution on $\Gcal(n)$ is given by $\mu_n\defeq\Prob^{S^*_n}$. This is a classical setting, in which for each dimension the same number of random variables are averaged.
\item Consider for all $n\in\N$ a collection of random variables
\[ X^{(n)}_{11},X^{(n)}_{12},\ldots,X^{(n)}_{1n_1},X^{(n)}_{21},\ldots,X^{(n)}_{2n_2},\ldots,X^{(n)}_{d1},\ldots,X^{(n)}_{dn_d},
 \]
 where $n_{\delta}=n_{\delta}(n)\in\N$ for all $n$, but this dependence is suppressed notationally. In particular, in each dimension $\delta$ we allow for a different number $n_{\delta}$ of random variables $X^{(n)}_{\delta 1},X^{(n)}_{\delta 2},\ldots,X^{(n)}_{\delta n_{\delta}}$, which we assume take values on the same grid $\Gcal_{\delta}= v_{\delta}+w_{\delta}\Z$. Again, admitting different scaling in each dimension, let $c(n)$ be a $d$-dimensional scaling vector, for example, $c_{\delta}(n) = (n_{\delta})^{-1/2}$ for all $\delta\in\{1,\ldots,d\}$. Define
 \[
S_n\defeq \left(\sum_{i=1}^{n_1} X^{(n)}_{1i},\ldots, \sum_{i=1}^{n_{d}} X^{(n)}_{1d}\right) \qquad \text{and} \qquad S_n^{*}\defeq c(n)\circ S_n.
\]
Now, $S_n^{*}$ assumes values on the grid 
\[
\Gcal(n) = v(n) + w(n)\circ\Z^d,
\] 
where 
\[
 v_{\delta}(n)=v_{\delta}\cdot n_{\delta}\cdot c_{\delta}(n)\quad\text{and}\quad w_{\delta}(n)=w_{\delta}\cdot c_{\delta}(n),  
\]
so again the distribution of interest is $\mu_n\defeq\Prob^{S_n^*}$.
\end{enumerate}

Before stating our main theorem, we explain some notation. If $a$ and $b$ are vectors, comparisons such as $a\leq b$ are always understood to be componentwise. In particular, the set $[a,b]\defeq\{x\in\R^d: a\leq x \leq b\}$ defines a multidimensional interval. If $I$ is a $d$-dimensional interval, $I=\prod I_{\delta}$, we define its length $\abs{I}\defeq (\abs{I_1},\ldots,\abs{I_d})$, where for an interval $J\subseteq\R, \abs{J}$ denotes its length. Now if $M\subseteq\R^d$ is any set, then we define
\[
\Ical(M) \defeq \{I\subseteq M, I \text{ is a $d$-dim.\ interval with }\abs{I}>0\}
\]
and if $c$ is a $d$-dimensional vector with $c>0$, we set
\[
\Ical_c(M) \defeq \{I\subseteq M, I \text{ is a $d$-dim.\ interval with } \abs{I}\geq c\}
\]
Our first theorem introduces a new notion of a local limit law as follows:

\begin{theorem}
\label{thm:intervallocallaw}
Assume that the local limit theorem \eqref{eq:locallaw} holds. Then let $[a,b]$ be a non-degenerate compact $d$-dimensional interval, so that there is a $c>0$ with $f\geq c$ on $[a,b]$. Let $(m(n))_n$ be a sequence of $d$-dimensional minimal length parameter vector satisfying 
\begin{equation}
\label{eq:mncondition}
\frac{m_{\delta}(n)}{w_{\delta}(n)} \longrightarrow\infty \text{ for all } \delta\in\{1,\ldots,d\}.	
\end{equation}
Then we obtain
\begin{equation}
\label{eq:intervallocallaw}	
\sup_{I\in\Ical_{m(n)}([a,b])}\bigabs{\frac{\mu_n(I)}{\mu(I)}-1} \xrightarrow[n\to\infty]{} 0.
\end{equation}
Further, if  $(m(n))_n$ is a sequence for which \eqref{eq:mncondition} does \emph{not} hold, also \eqref{eq:intervallocallaw} does not hold.
\end{theorem}
The interpretation of Theorem~\ref{thm:intervallocallaw} is that even when zooming in onto smaller and smaller intervals whose edge lengths may decay at a rate of order $\gg w_{\delta}(n)$, the measure $\mu_n$ is well-approximated by the limit distribution $\mu$. 
\begin{proof}
The proof of the first statement of the theorem is divided into three steps:\newline
\underline{Step 1:} In this step we show that the histogram estimator of the grid distribution $\mu_n$ converges relatively uniformly to $f$ on $[a,b]$.
To this end, we define the $d$-dimensional kernels $K_n(x)\defeq \frac{1}{\bar{w}(n)}\one_{(-w(n)/2,w(n)/2]}(x)$ for all $x\in\R^d$. The kernel is used to spread the probability mass of each grid point onto its surrounding box. We define the $d$-dimensional histogram of $\mu_n$ as the density
\[
\forall\,x\in\R: h_n(x) \defeq (K_n \ast \mu_n) (x) = \sum_{y\in\Gcal(n)}\frac{1}{\bar{w}(n)}\mu_n(y)\one_{\left(y-\frac{w(n)}{2},y+\frac{w(n)}{2}\right]}(x).
\]
Since $f$ is continuous and $f\geq c>0$ on $[a,b]$, we can find $a'<a$, $b'>b$ and $c'\in (0,c)$ such that $f\geq c'$ on $[a',b']$. Then there is an $N\in\N$ such that $w(N)/2 < \min(a-a',b'-b)$, where the minimum is taken componentwise. Note that the enlargement of $[a,b]$ to $[a',b']$ guarantees that for any $x\in[a,b]$ there is a $y\in[a',b']$ with $\abs{x-y}\leq w(n)/2$, where the absolute value is taken componentwise. Then we find for all $n\geq N$ and $x\in[a,b]$ that 
\begin{equation}
\label{eq:onlyonesummand}
\frac{h_n(x)}{f(x)} = \sum_{\substack{y\in\Gcal(n)\\ y\in[a',b']}}\frac{\frac{1}{\bar{w}(n)}\mu_n(y)}{f(y)}\frac{f(y)}{f(x)}\one_{\left(y-\frac{w(n)}{2},y+\frac{w(n)}{2}\right]}(x).
\end{equation}
We analyze the factors in the summands of \eqref{eq:onlyonesummand}: By \eqref{eq:locallaw}, 
\[
\frac{\frac{1}{\bar{w}}\mu_n(y)}{f(y)} \xrightarrow[n\to\infty]{} 1
\]
uniformly for $y$ over $\Gcal(n)\cap[a',b']$. In addition, since $f$ is uniformly continuous over $[a',b']$, for all $y\in [a',b']$ and $x\in [a,b]$ with $\abs{x-y}\leq w(n)/2$ componentwise, we find
\[
\bigabs{\frac{f(y)}{f(x)}-1} \leq \frac{\abs{f(y)-f(x)}}{c'}\leq \frac{\omega_f(\norm{w(n)/2}_1)}{c'},
\]
where $\omega_f(\cdot)$ denotes the modulus of continuity of $f$.
 Since $\omega_f(\epsilon)\to 0$ as $\epsilon\to 0$, the terms $f(y)/f(x)$ in \eqref{eq:onlyonesummand} converge to 1 uniformly or vanish due to the indicator. But noting that in the finite sum \eqref{eq:onlyonesummand}, only one summand survives for any given $x$, we have thus argued that
\begin{equation*}
\label{eq:firststep}
\sup_{x\in[a,b]}\bigabs{\frac{h_n(x)}{f(x)}-1} \xrightarrow[n\to\infty]{} 0,
\end{equation*}
which concludes the first step.\newline
\underline{Step 2:} In this step, we show that the continous version of $\mu_n$, $H_n\defeq h_n\lebesgue^d$, converges relatively uniformly to $\mu$ on arbitrary small intervals with positive length vector, that is, we show
\begin{equation}
\label{eq:secondstep}
\sup_{I\in\Ical([a,b])}\bigabs{\frac{H_n(I)}{\mu(I)}-1}	\xrightarrow[n\to\infty]{} 0.
\end{equation}
To see this, we obtain from the first step that there is a sequence $\epsilon_n\searrow 0$ such that for all $x\in[a,b]$: $h_n(x)/f(x)\in[1-\epsilon_n,1+\epsilon_n]$. Therefore, for any $I\in\Ical([a,b])$:
\[
\frac{H_n(I)}{\mu(I)} = \frac{\int_I \frac{h_n(x)}{f(x)}f(x)\lebesgue^d(\de x)}{\int_I f(x)\lebesgue^d(\de x)} \in [1-\epsilon_n,1+\epsilon_n].
\]
\underline{Step 3:} In this step, we show 
\begin{equation}
\label{eq:thirdstep}
\sup_{I\in \Ical_{m(n)}([a,b])}\bigabs{\frac{\mu_n(I)}{H_n(I)}-1} \xrightarrow[n\to\infty]{} 0.	
\end{equation}
To this end, we choose an interval $I\in\Ical_{m(n)}([a,b])$ arbitrarily. Then $\abs{I}\geq m(n)$, where $m(n)$ is a vector sequence such that $m_{\delta}(n)/w_{\delta}(n)\to\infty$ for all $\delta\in\{1,\ldots,d\}$. A quick thought yields that
\begin{equation}
\label{eq:howmanypoints}
\forall\, \delta\in\{1,\ldots,d\}:\ \#\left(\Gcal_{\delta}(n) \cap I_{\delta}\right) \geq \left\lfloor\frac{m_{\delta}(n)}{w_{\delta}(n)}\right\rfloor.
\end{equation}
For all $\delta\in\{1,\ldots,d\}$, let $l_{\delta}\in\Gcal_{\delta}(n)$ be the smallest and $r_{\delta}\in\Gcal_{\delta}(n)$ be the largest point of $\Gcal_{\delta}(n)$ so that $\{l_{\delta},l_{\delta}+w_{\delta}(n),l_{\delta}+2w_{\delta}(n),\ldots,r_{\delta}-w_{\delta}(n),r_{\delta}\}\subseteq I_{\delta}$. For any finite subset $\{k_1,\ldots, k_l\}\subseteq\Gcal_{\delta}(n)$, we define the union of disjoint intervals
\begin{equation}
\label{eq:unionofintervals}	
I_{\delta}(k_1,\ldots,k_l) \defeq \bigcup_{i=1}^l\left(k_i-\frac{w_{\delta}(n)}{2},k_i+\frac{w_{\delta}(n)}{2}\right].
\end{equation}
Note that if $\{k_1,\ldots,k_l\}$ is a discrete interval in $\Gcal_{\delta}(n)$, then $I_{\delta}(k_1,\ldots,k_l)$ is an interval in $\R$. Now we calculate
\begin{align}
&\frac{\abs{\mu_n(I) - H_n(I)}}{H_n(I)}\notag\\
&\leq \frac{\mu_n\left(
\bigcup_{\delta=1}^d \left(\left\{l_{\delta}-w_{\delta}(n),l_{\delta},r_{\delta},r_{\delta}+w_{\delta}(n) \right\}\times \prod_{\delta'\neq\delta}\left\{l_{\delta'}-w_{\delta'}(n),\ldots,r_{\delta'}+w_{\delta'}(n)\right\}\right)
\right)}
{\mu_n\left(\prod_{\delta=1}^d\left\{l_{\delta}+w_{\delta}(n),\ldots,r_{\delta}-w_{\delta}(n)\right\}\right)}\notag\\
&\leq \sum_{\delta=1}^d 
\frac{\mu_n\left(
 \left\{l_{\delta}-w_{\delta}(n),l_{\delta},r_{\delta},r_{\delta}+w_{\delta}(n) \right\}\times \prod_{\delta'\neq\delta}\left\{l_{\delta'}-w_{\delta'}(n),\ldots,r_{\delta'}+w_{\delta'}(n)\right\}
\right)}
{\mu_n\left(\prod_{\delta^*=1}^d\left\{l_{\delta^*}+w_{\delta^*}(n),\ldots,r_{\delta^*}-w_{\delta^*}(n)\right\}\right)}\notag\\
&= \sum_{\delta=1}^d 
\frac{
 H_n\left(I_{\delta}\left(l_{\delta}-w_{\delta}(n),l_{\delta},r_{\delta},r_{\delta}+w_{\delta}(n) \right)\times \prod_{\delta'\neq\delta}I_{\delta'}\left(l_{\delta'}-w_{\delta'}(n),\ldots,r_{\delta'}+w_{\delta'}(n)\right)\right)
}
{H_n\left(\prod_{\delta^*=1}^d I_{\delta^*}\left(l_{\delta^*}+w_{\delta^*}(n),\ldots,r_{\delta^*}-w_{\delta^*}(n)\right)\right)}
\label{eq:foursummands}
\end{align}
where in \eqref{eq:foursummands} we have $d$ summands, each of which can be massaged in the same way. Therefore, we pick a $\delta\in\{1,\ldots, d\}$ arbitrarily and inspect the $\delta$-th summand. In the numerator, we see that the set $I_{\delta}(\ldots)$ is a disjoint union of four intervals. It follows that the $\delta$-th summand can be written as a sum of four summands, which we will again treat the same way, the second summand being
\[
\frac{
 H_n\left(I_{\delta}\left(l_{\delta}\right)\times \prod_{\delta'\neq\delta}I_{\delta'}\left(l_{\delta'}-w_{\delta'}(n),\ldots,r_{\delta'}+w_{\delta'}(n)\right)\right)
}
{H_n\left(\prod_{\delta^*=1}^d I_{\delta^*}\left(l_{\delta^*}+w_{\delta^*}(n),\ldots,r_{\delta^*}-w_{\delta^*}(n)\right)\right)} =: \frac{H_n\left(I^{(1)}\right)}{H_n\left(I^{(2)}\right)}.
\]
Now we see
\begin{equation}	
\label{eq:threefactors}
\frac{H_n\left(I^{(1)}\right)}{H_n\left(I^{(2)}\right)} = 
\frac{H_n\left(I^{(1)}\right)}{\mu\left(I^{(1)}\right)}\cdot
\frac{\mu\left(I^{(1)}\right)}{\mu\left(I^{(2)}\right)}\cdot
\frac{\mu\left(I^{(2)}\right)}{H_n\left(I^{(2)}\right)}.
\end{equation}
By Step 2 -- after replacing $a$ and $b$ by $a'$ and $b'$ as in Step 1 -- we obtain that the first and third factor in \eqref{eq:threefactors} converge to $1$ uniformly. For the second factor, we calculate
\begin{align*}
&\frac{\mu\left(I^{(1)}\right)}{\mu\left(I^{(2)}\right)} = \frac{\int_{I^{(1)}}f\de\lebesgue^d}{\int_{I^{(2)}} f\de\lebesgue^d} \leq \frac{f_{\max}\cdot\lebesgue^d\left(I^{(1)}\right)}{f_{\min}\cdot\lebesgue^d\left(I^{(2)}\right)}\\
&=  \frac{f_{\max}\cdot w_{\delta}(n) \cdot \prod_{\delta'\neq\delta}\left(r_{\delta'}-l_{\delta'}+ 3 w_{\delta'}(n)\right)}{f_{\min}\cdot \lebesgue\left(I_{\delta}\left(l_{\delta}+w_{\delta}(n),\ldots,r_{\delta}-w_{\delta}(n)\right)\right)\cdot\prod_{\delta'\neq\delta}\left(r_{\delta'}-l_{\delta'}-w_{\delta'}(n)\right)}\\
&\leq\frac{f_{\max}\cdot w_{\delta}(n)\cdot 4^{d-1}\cdot}{f_{\min}\cdot \lebesgue\left(I_{\delta}\left(l_{\delta}+w_{\delta}(n),\ldots,r_{\delta}-w_{\delta}(n)\right)\right)} \leq \frac{f_{\max}\cdot w_{\delta}(n)\cdot 4^{d-1}}{f_{\min}\cdot \left(\lfloor m_{\delta}(n)/w_{\delta}(n)\rfloor -2\right)\cdot w_{\delta}(n)}.
\end{align*}
where $f_{\max}$ and $f_{\min}$ denote the maximum and minimum of $f$ over $[a',b']$ and we used \eqref{eq:howmanypoints}. In total, we obtain
\[
\sup_{I\in\Ical_{m(n)}([a,b])}\bigabs{\frac{\mu_n(I)}{H_n(I)}-1} \leq \sum_{\delta=1}^d 4\cdot \frac{f_{\max}\cdot  4^{d-1}}{f_{\min}\cdot(\lfloor m_{\delta}(n)/w_{\delta}(n)\rfloor -2)}\xrightarrow[n\to\infty]{} 0,
\]
which concludes the third step. Combining \eqref{eq:secondstep} and \eqref{eq:thirdstep} shows \eqref{eq:intervallocallaw}, which proves the first statement of the theorem.

 For the second statement, assume that \eqref{eq:mncondition} does not hold. This entails that there is a $\delta^*\in\{1,\ldots,d\}$, a $C\geq 0$ and a subsequence $(m_{\delta^*}(n)/ w_{\delta^*}(n))_{n\in J}$ for some $J\subseteq\N$ such that for all $n\in J: m_{\delta^*}(n)/w_{\delta^*}(n)\leq C$. It now suffices to construct for each $n\in J$ an interval $I^{(n)}\subseteq[a,b]$ with $\abs{I^{(n)}}\geq m(n)$, and a real number $\beta<1$ such that
 \[
 \frac{\mu_n(I^{(n)})}{\mu(I^{(n)})}\leq \beta \qquad \forall\,n\in J.
 \]
 To this end, let $l\defeq \lceil C\rceil$ (w.l.o.g.\ $l\geq 2$ after enlarging $C$) and note that for all $n\in J$: $l\geq \lceil m_{\delta^*}(n)/w_{\delta^*}(n) \rceil$. Now for all $n\in J$ pick an arbitrary discrete interval $\{x^{(n)}_0,\ldots,x^{(n)}_l\}\subseteq \Gcal_{\delta^*}(n)\cap [a_{\delta^*},b_{\delta^*}]$, and define the open interval $I^{(n)}_{\delta^*}\defeq (x^{(n)}_0,x^{(n)}_l)\subseteq [a_{\delta^*},b_{\delta^*}]$.
 
 For all $\delta\neq\delta^*$ we pick intervals $I^{(n)}_{\delta}\subseteq [a_{\delta},b_{\delta}]$ arbitrary with $\abs{I^{(n)}_{\delta}}\geq m_{\delta}(n)$ and with the condition that these have the form (cf.\,\eqref{eq:unionofintervals})
 \begin{equation}
 \label{eq:niceintervals}
 I^{(n)}_{\delta} = I\left(y^{(n)}_{1,\delta},\ldots,y^{(n)}_{k,\delta}\right)
 \end{equation}
for some discrete interval $\left\{y^{(n)}_{1,\delta},\ldots,	y^{(n)}_{k,\delta}\right\}\subseteq [a_{\delta},b_{\delta}]$.
Then since
 \[
 \abs{I^{(n)}_{\delta^*}} = l\cdot w_{\delta^*}(n) \geq \left\lceil\frac{m_{\delta^*}(n)}{w_{\delta^*}(n)}\right\rceil \cdot w_{\delta^*}(n) \geq m_{\delta^*}(n), \quad \text{we have} \quad \abs{I^{(n)}}\geq m(n).
 \]
In addition, we note that
 \begin{align*}
 &\mu_n(I^{(n)}) = \mu_n\left(\left\{x^{(n)}_1,\ldots,x^{(n)}_{l-1}\right\}\times \prod_{\delta\neq\delta^*}I^{(n)}_{\delta}\right)\\
 &\text{and}\quad H_n(I^{(n)}) = \mu_n\left(\left\{x^{(n)}_1,\ldots,x^{(n)}_{l-1}\right\}\times \prod_{\delta\neq\delta^*}I^{(n)}_{\delta}\right) + \frac{1}{2} \mu_n\left(\left\{x^{(n)}_0,x^{(n)}_l\right\}\times \prod_{\delta\neq\delta^*}I^{(n)}_{\delta}\right),
 \end{align*}
 where we used \eqref{eq:niceintervals} for the second equality.
 Since $h_n/f\to 1$ uniformly on $[a,b]$ and $0<c\leq f\leq \bar{c}$ on $[a,b]$ for some constant $\bar{c}$, we find $c/2 \leq h_n \leq 2\bar{c}$ on $[a,b]$ for finally all $n\in J$. For all such $n$ we may calculate, using that for any $y>0$, $0\leq x\mapsto x/(x+y))$ is isotonic, denoting by $h_n^{\min}$ resp.\ $h_n^{\max}$ the minimum resp.\ maximum of $h_n$ over $[a,b]$ and setting $P\defeq \prod_{\delta\neq \delta^*} I^{(n)}_{\delta}$:
 \begin{align*}
 \frac{\mu_n(I^{(n)})}{H_n(I^{(n)})}&= \frac{H_n\left(I_{\delta}(x^{(n)}_1,\ldots,x^{(n)}_{l-1})\times P\right)}{H_n\left(I_{\delta}(x^{(n)}_1,\ldots,x^{(n)}_{l-1})\times P\right) +\frac{1}{2}H_n\left(I_{\delta}(x^{(n)}_0)\times P\right) + \frac{1}{2}H_n\left(I_{\delta}(x^{(n)}_{l})\times P\right)}\\
 &\leq \frac{h_n^{\max}\cdot (l-1)w_{\delta}(n)\cdot\lebesgue^{d-1}(P)}{h_n^{\max}\cdot (l-1)w_{\delta}(n)\cdot\lebesgue^{d-1}(P) + 2\cdot\frac{1}{2} h^{\min}_n\cdot w_{\delta}(n)\cdot\lebesgue^{d-1}(P)}\\
 &= \frac{1}{1+\frac{h_n^{\min}}{h_n^{\max}(l-1)}} \ \leq\ \frac{1}{1+\frac{c/2}{2\bar{c}(l-1)}}.
 \end{align*}
 But now for all $n\in J$,
 \[
 \frac{\mu_n(I_n)}{\mu(I_n)} = \frac{\mu_n(I_n)}{H_n(I_n)} \cdot \frac{H_n(I_n)}{\mu(I_n)},
 \]
 where the first factor on the r.h.s.\ is bounded by $(1+\frac{c}{4\bar{c}(l-1)})^{-1}<1$ for all $n\in J$ large enough, and the second factor converges to $1$ as $n\to\infty$ in $J$, due to Step 2 above. This finishes the proof.
\end{proof}

We isolate the i.i.d.\ univariate case as a corollary, since it is the most prominent setting for local limit theorems:

\begin{corollary}
\label{cor:iidunivariate}
Let $X$ be lattice distributed on $v+w\Z$ with maximal span $w$. We assume $\mu\defeq\E X$ and $\sigma^2\defeq\V X>0$ exist. Let $(X_i)_{i\in\N}$ be i.i.d.\ copies of $X$. Let $(m_n)_n$ be a sequence with $m_n\sqrt{n}\to\infty$, for example, $m_n = n^{-1/2+\epsilon}$ for some $\epsilon\in(0,1/2)$ or $m_n = \ln(\ln(n))/\sqrt{n}$. Denote by $\Ncal(0,1)$ the standard normal distribution on $\R$, it then holds for all $a<b\in\R:$
\[
\sup_{I\in\Ical_{m_n}([a,b])}\bigabs{\frac{\frac{\sigma\sqrt{n}}{w}\Prob\left(\frac{1}{\sigma\sqrt{n}}\sum_{i=1}^n (X_i-\mu) \in I \right)}{\Ncal(0,1)(I)}-1}\xrightarrow[n\to\infty]{} 0.
\]
\end{corollary}
\begin{proof}
This is a direct corollary of Theorem~\ref{thm:intervallocallaw}, take $\mu_n$ as the distribution of $\frac{1}{\sigma\sqrt{n}}\sum_{i=1}^n (X_i-\mu)$, $\mu$ as the standard normal distribution, and $\Gcal(n)$ as the grid $v\sqrt{n}/\sigma + w/(\sigma\sqrt{n})\Z$. For this setting, the local limit theorem \eqref{eq:locallaw} is well known, see \cite{Durrett:2019}.
\end{proof}

Next, we move from the study of the i.i.d.\ univariate case to an example of a correlated multivariate case. The local limit theorem we will use to obtain the interval type local limit theorem was recently derived in \cite{Fleermann:Kirsch:Toth:2020:a}.

For each $n\in\N$ with $n\geq d$ we consider a collection of $\{\pm 1\}$-valued random variables ("spins")
\[ 
X^{(n)} = \left(X^{(n)}_{11},X^{(n)}_{12},\ldots,X^{(n)}_{1n_1},X^{(n)}_{21},\ldots,X^{(n)}_{2n_2},\ldots,X^{(n)}_{d1},\ldots,X^{(n)}_{dn_d}\right)
\]
where $n_{\delta}=n_{\delta}(n)\in\N$ for all $\delta\in\{1,\ldots,d\}$ and $n_1+\ldots+n_d =n$. In other words, for each $n\geq d$ we have a family of $n$ random variables which is subdivided into $d$ groups.  As $n\to\infty$, we assume that $\frac{n_{\delta}}{n}\to \alpha_{\delta}\in[0,1]$, which is the asymptotic fraction of group-$\delta$ variables within $X^{(n)}$. For each $n\in\N$, the distribution of $X^{(n)}$ is given by
\begin{equation}
\label{eq:MultCW}	
\Prob(X^{(n)}_{11}=x_{11},\ldots,X^{(n)}_{dn_d}=x_{dn_d}) \propto \exp\left(\frac{1}{2n}\sum_{\delta,\gamma=1}^{d}J_{\delta,\gamma}\sum_{i=1}^{n_{\delta}}\sum_{j=1}^{n_{\gamma}} x_{\delta i}x_{\gamma j}\right), 
\end{equation}
where $J$ is a $d\times d$ matrix called \emph{coupling matrix}, whose entries $J_{\delta,\gamma}$ are called \emph{coupling constants} and describe the correlation within and between spins in groups $\delta$ and $\gamma$. This probability model is called \emph{multi-group Curie-Weiss model} \cite{Contucci:Ghirlanda:2007, Kac:1968, Toth:2019, Kirsch:2007}. 
Now we define
\begin{equation}
\label{eq:CWSNstar}
S^*_n\defeq \left(\frac{1}{\sqrt{n_1}}\sum_{i=1}^{n_1} X^{(n)}_{1i},\ldots, \frac{1}{\sqrt{n_d}}\sum_{i=1}^{n_{d}} X^{(n)}_{1d}\right),
\end{equation}
which assumes values on the grid 
\[
\Lcal(n) = \prod_{\delta=1}^d \left(\sqrt{n_{\delta}}+ \frac{2}{\sqrt{n_{\delta}}}\Z\right).
\]
The model in \eqref{eq:MultCW} is defined for two classes of coupling matrices $J$: Either all entries of  $J$ are equal to a constant $\beta\geq 0$, which is then called the \emph{homogeneous case}, or $J$ is an arbitrary positive definite matrix, which is then called the \emph{heterogeneous case}. In each case, we distinguish between three regimes, a high-temperature regime, a critical regime and a low-temperature regime. For the high-temperature regime -- characterized by $\beta\in[0,1)$ for the homogeneous case and by the condition that $J^{-1} - A$, where $A\defeq\diag(\alpha_1,\ldots,\alpha_d)$, is positive definite for the heterogeneous case -- $S_n^*$ converges in distribution to the $d$-dimensional normal distribution $\Ncal(0,C)$, where 
\[
C =
\begin{cases}
I_d + \sqrt{A} (\frac{\beta}{1-\beta})_{d\times d}\sqrt{A} & \text{in the homogeneous case,}\\
I_d + \sqrt{A}(J^{-1} - A)^{-1}\sqrt{A} & \text{in the heterogeneous case},
\end{cases}
\]
see \cite{Toth:2019}. In \cite{Fleermann:Kirsch:Toth:2020:a} it was shown that this convergence even holds locally, that is,
\begin{equation}
\label{eq:CWlocal}
	\sup_{x\in\Lcal(n)} \bigabs{\frac{\prod_{\delta=1}^d\sqrt{n_{\delta}}}{2^d}\Prob(S_n^*=x) - \varphi_C(x)}\xrightarrow[n\to\infty]{} 0,
\end{equation}
where $\varphi_C$ denotes the $\lebesgue^d$-density of $\Ncal(0,C)$. We now obtain the following interval type local limit theorem:

\begin{corollary}
\label{cor:CW}
Let $S_n^*$ be as in \eqref{eq:CWSNstar}, where the random variables stem from the multi-group Curie-Weiss-model as in \eqref{eq:MultCW} with coupling matrix $J$ chosen from the high-temperature regime of either the homogeneous or the heterogeneous case. Now let $a<b\in\R^d$ be arbitrary and $m(n)$ be a d-dimensional sequence satisfying $m_\delta(n)\sqrt{n_{\delta}}\to\infty$ for all $\delta\in\{1,\ldots,d\}$ as $n\to\infty$. Then we obtain
\[
\sup_{I\in\Ical_{m(n)}([a,b])} \bigabs{\frac{\frac{\prod_{\delta}\sqrt{n_\delta}}{2^d}\Prob\left(S_n^*\in I\right) }{\Ncal(0,C)(I)}-1} \xrightarrow[n\to\infty]{} 0.
\]
\end{corollary}
\begin{proof}
This is a direct consequence of \eqref{eq:CWlocal} and Theorem~\ref{thm:intervallocallaw}.
\end{proof}

For the validity of Theorem~\ref{thm:intervallocallaw}, the diameter of the invervals in consideration must not decrease too quickly. As it turns out, this is purely due to the lattice type distributions $\mu_n$. The next theorem shows that in the Lebesgue continuous case, \eqref{eq:intervallocallaw} holds for intervals of any positive length. 

\begin{theorem}
\label{thm:continuousintervallocal}
Let $f$ and $(f_n)_n$ be a probability density functions on $\R^d$, $\mu=f\lebesgue^d$ and $\mu_n=f_n\lebesgue^d$ for all $n\in\N$. Then if $\mu_n$ converges locally weakly to $\mu$, that is,
\begin{equation}
\label{eq:continuouslocalweak}
\sup_{x\in\R^d}\abs{f_n(x)-f(x)} \xrightarrow[n\to\infty]{} 0,	
\end{equation}
also the following interval type local limit theorem holds:
If $[a,b]\subseteq\R^d$ is a non-degenerate $d$-dimensional interval, so that there is a $c>0$ such that $f\geq c$ on $[a,b]$, then
\begin{equation}
\label{eq:continuousintervallocal}
\sup_{I\in\Ical([a,b])} \bigabs{\frac{\mu_n(I)}{\mu(I)}-1}\xrightarrow[n\to\infty]{} 0.
\end{equation}
\end{theorem}
\begin{proof}
From \eqref{eq:continuouslocalweak} and the fact that $f\geq c>0$ on $[a,b]$, we obtain
\[
\sup_{x\in[a.b]}\bigabs{\frac{f_n(x)}{f(x)}-1} =:\epsilon_n \xrightarrow[n\to\infty]{} 0.
\]	
Then for $I\in\Ical([a,b])$ arbitrary we observe
\[
\frac{\mu_n(I)}{\mu(I)} = \frac{\int_I f_n \de \lebesgue^d}{\int_I f \de\lebesgue^d} = \frac{\int_I \frac{f_n}{f}f \de \lebesgue^d}{\int_I f \de\lebesgue^d}\in [1-\epsilon_n, 1+\epsilon_n],
\]
which shows the statement.
\end{proof}

Again, we isolate the i.i.d.\ univariate case as a corollary. We denote by $\Ncal(0,1)$ the standard normal distribution and by $\varphi$ its $\lebesgue$-density.

\begin{corollary}
\label{thm:continuouslocallaw}
Let $X$ be a real valued random variable with existing expectation $\mu\defeq\E X\in\R$ and variance $\sigma^2\defeq\V X>0$. Further, we assume that $\Prob^X=f\lebesgue$ for some density function $f:\R\to\R$. Let $(X_n)_n$ be i.i.d.\ copies of $X$. Then denote by $f_n$ the (existing) $\lebesgue$-density of
\[
\frac{1}{\sigma\sqrt{n}}\sum_{i=1}^n(X_i-\mu).
\]
\begin{enumerate}
\item[a)] The local limit theorem
\begin{equation}
\label{eq:continuouslocallaw}
\sup_{x\in\R}\bigabs{f_n(x)-\varphi(x)}\xrightarrow[n\to\infty]{} 0
\end{equation}
holds if and only there is an $N\in\N$ such that $f_N$ is a bounded function. In particular, \eqref{eq:continuouslocallaw} holds whenever $f$ is bounded.
\item[b)] If \eqref{eq:continuouslocallaw} holds, so does the following interval type local limit theorem:
\begin{equation*}
\forall\, a<b\in\R: \quad \sup_{I\in \Ical([a,b])} \bigabs{\frac{\Prob\left(\frac{1}{\sigma\sqrt{n}}\sum_{i=1}^n(X_i-\mu)\, \in\, I\right)}{\Ncal(0,1)(I)}-1}\xrightarrow[n\to\infty]{} 0.
\end{equation*}
\end{enumerate}	
\end{corollary}
\begin{proof}
Statement a) is well-known, see Theorem 7 in \cite[198]{Petrov:1975}. Statement b) follows immediately from Theorem~\ref{thm:continuousintervallocal}.
\end{proof}

\sloppy

\vspace{1cm}

\noindent\textsf{(Michael Fleermann, Werner Kirsch and Gabor Toth)\newline
FernUniversit\"at in Hagen\newline
Fakult\"at f\"ur Mathematik und Informatik\newline 
Universit\"atsstra\ss e 1\newline 
58084 Hagen}\newline
\textit{E-mail addresses:}\newline
\texttt{michael.fleermann@fernuni-hagen.de}\newline
\texttt{werner.kirsch@fernuni-hagen.de}\newline
\texttt{gabor.toth@fernuni-hagen.de}
\vspace{1cm}

\end{document}